\documentclass[10pt, reqno]{amsart}
\usepackage{amsmath, amsthm, hyperref}
\usepackage{times}
\usepackage{amsfonts}
\usepackage{amssymb}
\usepackage{mathtools}
\mathtoolsset{showonlyrefs}
\usepackage{xcolor} 
\usepackage{graphicx} 
\usepackage{cite} 
\usepackage{enumitem}
\setenumerate{itemsep=5pt} 
\setitemize{itemsep=5pt, leftmargin=10pt} 
\setlist[enumerate]{leftmargin=*, itemsep=5pt} 

\renewcommand{\vec}[1]{\mathbf{#1}}
\newcommand{\D}{\mathrm{D}}

\newcommand{\M}{\mathsf{M}}

\newcommand{\N}{\mathrm{N}}
\newcommand{\Q}{\mathbb{Q}}

\renewcommand{\S}{\mathcal{S}}

\newcommand{\Z}{\mathbb{Z}}

\newcommand{\K}{\mathbb{K}}
\renewcommand{\O}{\mathcal{O}}

\newcommand{\RR}{\mathcal{R}}
\newcommand{\ran}{\operatorname{ran}}
\newcommand{\adj}{\operatorname{adj}}
\newcommand{\ind}{\operatorname{ind}}

\newcommand{\mdim}{\operatorname{mdim}}

\newcommand{\megamatrix}[9]{\begin{bmatrix} #1 & #2 & #3 \\ #4 & #5 & #6\\ #7 & #8 & #9\end{bmatrix}}

\newcommand{\minimatrixsmall}[4]{[\begin{smallmatrix} #1 & #2\\ #3 & #4 \end{smallmatrix}]}

\newcommand{\diag}{\operatorname{diag}}
\newcommand{\rank}{\operatorname{rank}}

\newcommand{\pdet}{\operatorname{pdet}}

\newtheorem{thm}{Theorem}[section]
\newtheorem{prop}[thm]{Proposition}
\newtheorem{cor}[thm]{Corollary}

\newtheorem{lemma}[thm]{Lemma}

\theoremstyle{definition}
\newtheorem{ex}[thm]{Example}
\newtheorem{rem}[thm]{Remark}
\newtheorem{prob}[thm]{Problem}

\allowdisplaybreaks
\begin{document}
\title[Power-integral matrices over number fields]{Power-integral matrices over number fields: the Drazin inverse, pseudo-determinant, and numerical semigroups}

\author{Theo Chinn}
\address{Department of Mathematics and Statistics, Pomona College, 610 N. College Ave., Claremont, CA 91711, USA}
\email{twcz2023@mymail.pomona.edu}

\author[Junshu Feng]{Junshu Feng}
\email{jfwv2022@mymail.pomona.edu}

\author[S.R.~Garcia]{Stephan Ramon Garcia}
\email{stephan.garcia@pomona.edu}
\urladdr{\url{https://stephangarcia.sites.pomona.edu/}}

\author[Hechun Zhang]{Hechun Zhang}
\address{Department of Mathematical Sciences, Tsinghua University, Beijing,  China}
\email{hczhang@tsinghua.edu.cn}

\thanks{SRG was partially supported by NSF grant DMS-2054002. Hechun Zhang was partially supported by National Natural Science Foundation of China grants No. 12031007 and No. 11971255}



\begin{abstract}
We investigate matrices with entries in a number field such that some positive power has all its entries in the corresponding ring of integers.  Our work generalizes previous results in several directions and we find applications to numerical semigroups.
\end{abstract}

\maketitle

\section{Introduction}
In what follows, let $\M_d(\cdot)$ denote the set of $d \times d$ matrices with entries in the indicated set.
Let $\K$ denote a number field with ring of integers $\O_{\K}$.  Then $A \in \M_d(\K)$ is 
\emph{power integral with respect to $\K$} if there exists an $n \geq 1$ such that $A^n \in \M_d(\O_{\K})$.
For example, if $\K = \Q(\sqrt[3]{2})$ and
\begin{equation}\label{eq:FirstExample}
A=\begin{bmatrix}
            0 & \frac{1}{91} & -\frac{1}{91} \\[3pt]
            1& -1-\sqrt[3]{2}-\sqrt[3]{4} & 0 \\[1pt]
            1 & -93-2\sqrt[3]{2}& 1+\sqrt[3]{2}-\sqrt[3]{4}
        \end{bmatrix}\in \M_3(\K),
\end{equation}
then the first $n \geq 1$ such that $A^n\in \M_3(\O_\K)$ is $n=9{,}040{,}151{,}688$; see Example \ref{Example:ExplainFirst} for an explanation. In the setting $\K = \Q$, power-integral matrices appeared in \cite{PRM, christofides2025powers},
although the term does not seem to have been used prior to \cite{NSRM1}, which studied them in the context
of numerical semigroups \cite{Assi, Rosales}.  This theme was further developed in \cite{NSRM2,NSRM3,NSRM4}.

The results of this paper generalize previous work in multiple aspects.  First, we work over arbitrary number fields.  In this context,
$\K$ and $\O_{\K}$ take the place of $\Q$ and $\Z$.  Moreover, whereas many of the results in \cite{PRM, NSRM1} require the matrix to be invertible, we can now handle the singular case via the use of the Drazin inverse \cite{BenIsrael}, a generalized inverse familiar to practitioners of applied matrix theory.  
A relative of the determinant, the \emph{pseudo-determinant}, helps determine whether a matrix is power integral (Theorem \ref{Theorem:Unit}). It is the product of the nonzero eigenvalues of a matrix \cite{Knill, Holbrook}; see also \cite[p.~529]{Florescu}.

We also study connections to numerical semigroup theory.  For each $A \in \M_d(\K)$, we consider the 
additive subsemigroup
$\S_{\K}(A) = \{ n \in \Z_{\geq 0} : A^n \in \M_d(\O_{\K})\}$
of $\Z_{\geq 0}$.  This is the number-field generalization of the exponent semigroups considered in \cite{NSRM1,NSRM2,NSRM3,NSRM4}.
Clearly a matrix $A$ is power integral over $\K$ if and only if $\S_{\K}(A) \neq \{0\}$.  We prove, for example,
that if $A \in \M_d(\K)$ has pseudo-determinant in the unit group $\O_{\K}^{\times}$, then $\S_{\K}(A) \subseteq \S_{\K}(T)$ (Theorem \ref{Theorem:Contain}), in which $T = A^2A^{\D}$ is the index-$1$ component of $A$ \cite[Thm.~4.11, p.~169]{BenIsrael}; here $A^{\D}$ denotes the Drazin inverse.  This result can fail if the pseudo-determinant is not a unit (Example \ref{Example:ContainFail}). The relevant definitions appear in Section \ref{Section:Drazin}.

This paper is organized as follows.  Section \ref{Section:Drazin} covers preliminaries about the Drazin inverse and pseudo-determinant.
We consider algebraic integral matrices in Section \ref{Section:AlgebraicIntegral},
a key stepping stone toward power-integral matrices, in which the Drazin inverse plays a crucial role.  Section \ref{Section:PowerIntegral} contains our main results on power-integral matrices.  We wrap up in Section \ref{Section:Numerical} with connections to numerical semigroup theory and several open questions.

\section{The Drazin inverse and pseudo-determinant}\label{Section:Drazin}
Recall that the \emph{index} $\ind A$ of a matrix $A \in \M_d(\K)$ is defined as follows \cite[Def.~12.1.4]{SCLA}.
If $A$ is invertible, its index is $0$. If $A$ is not invertible,
its index is the least positive integer $k$ such that $\rank A^k = \rank A^{k+1}$; in particular,
\begin{equation}\label{eq:Fitting}
 \K^d = \ran(A^k)\oplus \ker(A^k) .
\end{equation}
The index of $A$ is the multiplicity of $0$ as a zero of the minimal polynomial $m_A(x)$.
Here $\ran(A^k)$ and $\ker(A^k)$ denote the range (column space) and kernel (nullspace) of $A^k$.

In what follows, we require the \emph{Drazin inverse}\label{p:Drazin}\footnote{The standard notation for the Drazin inverse of $A$ is $A^{\D}$; we employ $\hat{A}$ for typographical convenience.}
$\hat{A}$ of a square matrix $A$ \cite[Sec.~4.6]{BenIsrael}.
If $\ind A = 0$, define $\hat{A} = A^{-1}$.
If $\ind A = k \geq 1$, then the restriction $R$ of $A$ to $\ran(A^k)$ is invertible, so we define $\hat{A} \in \M_d(\K)$ by
\begin{equation*}
\hat{A}\vec{x} =
\begin{cases}
R^{-1} \vec{x} & \text{if $\vec{x} \in \ran(A^k)$},\\
\vec{0} & \text{if $\vec{x} \in \ker(A^k)$}.
\end{cases}
\end{equation*}
The \emph{Drazin inverse} $\hat{A}$ of the index-$k$ matrix $A$ satisfies
\begin{equation}\label{eq:Drazin}
A \hat{A} A^k = A^k, \qquad
A\hat{A}=\hat{A}A, \quad\text{and}\quad \hat{A}A\hat{A} = \hat{A}.
\end{equation}
It is the unique matrix satisfying these identities \cite[Thm.~4.7, p.~164]{BenIsrael}; moreover, $\ind \hat{A} \leq 1$ \cite[Thm.~4.9, p.~168]{BenIsrael}.
Not only does $\hat{A}$ commute with $A$, it is a polynomial in $A$.
Indeed, if $\ind A = k$, write the minimal polynomial of $A$ as 
\begin{equation}\label{eq:FactorMinimal}
m_A(x) = cx^k(1-xq(x)),
\end{equation}
in which $c \neq 0$ and $q(x) \in \K[x]$.  Then \cite[eq.~(42), p.~164; Ex.~34, p.~167]{BenIsrael} ensures that
\begin{equation}\label{eq:PolynomialHat}
\hat{A} = A^k (q(A))^{k+1}.
\end{equation}

The following result is implicit in Wedderburn's book \cite{Wedderburn}.  Ben-Israel and Greville call it 
the ``Index $1$-Nilpotent Decomposition'' 
\cite[Thm.~4.11, p.~169]{BenIsrael}.

\begin{thm}\label{Theorem:IndexNilpotent}
Let $A \in \M_d(\K)$.  Then there exist unique $T,N \in \M_d(\K)$ such that 
$A = T+N$, $\ind T \leq 1$, $N$ is nilpotent, and $NT = TN = 0$.  
\end{thm}

In fact, the \emph{index-$1$ component} $T$ is the Drazin inverse of $\hat{A}$.
Since $A\hat{A}$ is the projection onto $\ran(A^k)$ along $\ker(A^k)$ \cite[Cor.~2.7, p.~60; Ex.~30, p.~167]{BenIsrael},
we also have $T = A^2\hat{A}$.  Since $\hat{A}$ is a polynomial in $A$, the matrices $A$, $N$, and $T$ commute.

The Drazin inverse respects similarity: $A = QBQ^{-1}$ implies that $\hat{A} = Q\hat{B}Q^{-1}$ \cite[Ex.~26, p.~166]{BenIsrael}.
With respect to the decomposition \eqref{eq:Fitting}, the idempotent $A\hat{A}$ assumes the form $I \oplus 0$.  Since $A\hat{A}$, being a polynomial in $A$, commutes with $A$, we obtain the next theorem, which is a variant of \cite[Thm.~4.8, p.~164]{BenIsrael} (via rational canonical form instead of Jordan form) and a block-matrix version of Theorem \ref{Theorem:IndexNilpotent}.
Let $\diag(\cdot,\cdot)$ denote a block-diagonal matrix with the specified
square matrices as blocks on the diagonal.

If $r=0$ or $r=d$, then we omit the summand $Z$ or $E$,
respectively, in what follows.

\begin{thm}\label{Theorem:RatForm}
Let $A \in \M_d(\K)$ with $p_A(x) \in \O_{\K}[x]$ and let $0\leq r \leq d$ denote the
algebraic multiplicity of $0$ as an eigenvalue of $A$.
Then there exists an $S \in \M_d(\O_{\K})$,
an $E \in \M_{d-r}(\O_{\K})$ with index $0$, and a nilpotent $Z \in \M_r(\O_{\K})$ such that
$A = S\diag(E,Z)S^{-1}$.  
\end{thm}

\begin{proof}
Factor $p_A(x) = x^r q(x)$. Since $\O_{\K}$ is integrally closed \cite[Prop.~14, Sec.~16.3]{dummit2004abstract}, the monic factors of $p_A(x)$ over $\K$ belong to $\O_{\K}[x]$ \cite[Exercise~11, Sec.~15.3]{dummit2004abstract} (note that Gauss' lemma \cite[Prop.~5, Sec.~9.3]{dummit2004abstract} does not apply since $\O_{\K}$ need not be a unique factorization domain).
The asserted decomposition follows from the elementary-divisor version of the rational canonical form \cite{friedberg_linear_2003} since every companion matrix that arises has entries in $\O_{\K}$.
The similarity matrix $S$ from the rational canonical form can be multiplied by a suitable natural number to ensure that its entries belong to $\O_{\K}$.  \end{proof}

The uniqueness portion of Theorem \ref{Theorem:IndexNilpotent} ensures that $T = S\diag(E,0)S^{-1}$ and $N = S \diag(0,Z)S^{-1}$.  We define the \emph{pseudo-determinant} of $A$ to be $\pdet A = \det E$;
that is, it is the product of the nonzero eigenvalues of $A$. Thus, $\pdet A = 1$ for a nilpotent matrix \cite[p.~525]{Knill}, although this special case does not concern us. Note that 
\begin{equation*}
p_A(x) = x^r p_E(x) = p_T(x),
\end{equation*}
so $A$ and $T$ have the same eigenvalues and algebraic multiplicities.  Unlike its more famous namesake, the pseudo-determinant is not multiplicative \cite[p.~524]{Knill}.  However, 
\begin{equation}\label{eq:Multiplicative}
\pdet(A^n) = (\pdet A)^n
\end{equation}
for $n \geq 1$ \cite[Prop.~2.(10)]{Knill}.  This follows from the
fact that the eigenvalues of $A^n$ are the $n$th powers of the eigenvalues of $A$, multiplicity respected \cite[Cor.~11.1.4]{SCLA}.

The next lemma relates the Drazin inverse of an integral matrix to its pseudo-determinant. 
It is relevant to our study of power-integral matrices and exponent semigroups.

\begin{lemma}\label{Lemma:IntegralDrazin}
If $A\in\M_d(\O_{\K})$ and $\pdet A\in\O_{\K}^{\times}$, then $\hat{A}\in\M_d(\O_{\K})$.
\end{lemma}

\begin{proof}
Write $A=S\diag(E,Z)S^{-1}$ as in Theorem \ref{Theorem:RatForm}.  Then
$m_A(x)=x^kg(x)$, in which $k=\ind A$ and $g(0)\neq0$.
Since $g$ is a monic factor of $p_E$ and $\O_{\K}$ is integrally closed, $g\in\O_{\K}[x]$.
Moreover, $g(0)$ divides $p_E(0)=(-1)^{d-r}\det E$ in $\O_{\K}$, so $g(0)\in\O_{\K}^{\times}$.
Thus,
\begin{equation*}
q(x)=\frac{1-g(0)^{-1}g(x)}{x}
\in\O_{\K}[x],
\end{equation*}
and hence \eqref{eq:PolynomialHat} ensures that
\begin{equation*}
\hat{A}=A^k(q(A))^{k+1}\in\M_d(\O_{\K}). \qedhere
\end{equation*}
\end{proof}

\section{Algebraic-integral matrices}\label{Section:AlgebraicIntegral}
In our study of power-integral matrices, we must first consider a related topic.
We say that $A\in \M_d(\K)$ is \emph{algebraic integral} over $\K$ if there exists a monic $f(x)\in\O_{\K}[x]$ such that $f(A)=0$.  The following result generalizes \cite[Thm.~3.3]{NSRM1}, which established the equivalence of (a) -- (f) for $\K = \Q$.  The inclusion of (g) is novel even in that setting.

\begin{thm}\label{Theorem:AlgebraicIntegral}
Let $A\in \M_d(\K)$. The following are equivalent.
\begin{enumerate}
\item $A$ is algebraic integral over $\K$.
\item $m_A(x)\in\O_{\K}[x]$.
\item $p_A(x)\in\O_{\K}[x]$.
\item Every eigenvalue of $A$ is an algebraic integer.
\item There exists an integer $m \geq 1$ such that $m A^n\in \M_d(\O_{\K})$ for all $n\geq 0$.
\item There exist a $B\in \M_d(\O_{\K})$ and $S\in\M_d(\O_{\K})$ with $\det S \neq 0$ such that $A=SBS^{-1}$.
\item $T = A^2\hat{A}$ is algebraic integral over $\K$.
\end{enumerate}
\end{thm}

\begin{proof}
$(a) \Rightarrow (d)$ Let $f(A) = 0$, in which $f(x)\in\O_{\K}[x]$ is monic. Then every eigenvalue of $A$ is a zero of $f$ \cite[Thm.~9.3.3]{SCLA}. Since $f(x)$ is a monic polynomial whose coefficients are algebraic integers, each eigenvalue of $A$ is an algebraic integer \cite[Thm.~2.10]{AlgNumTheory}.

\medskip\noindent$(d) \Rightarrow (c)$ The hypothesis ensures that the zeros of $p_A(x)\in\K[x]$ are algebraic integers. Recall that sums and products of algebraic integers are algebraic integers \cite[Thm.~2.9]{AlgNumTheory}. The coefficients of $p_A(x)$ are elementary symmetric polynomials of the zeros of $p_A(x)$, so they are algebraic integers. Thus, the coefficients of $p_A(x)$ are in $\O_{\K}$.

\medskip\noindent$(c) \Rightarrow (b)$ Suppose that $p_A(x) \in \O_{\K}[x]$.  Since $m_A$ divides $p_A$ in $\K[x]$ and $m_A(x)$ is monic, we have $m_A(x)\in\O_{\K}[x]$; see the proof of Theorem \ref{Theorem:RatForm}.

\medskip\noindent$(b) \Rightarrow (a)$ By definition, $m_A(A) = 0$, so $A$ is algebraic integral over $\K$.

\medskip\noindent$(c) \Rightarrow (e)$
Suppose that $p_A(x)=x^d+c_{d-1}x^{d-1}+\cdots+c_1x+c_0\in\O_{\K}[x]$. Let $m$ be a positive integer such that $mA^i\in\M_d(\O_{\K})$ for each $i=1,2,\ldots, d-1$. The Cayley--Hamilton theorem ensures that $p_A(A)=0$, so $A^{d+i}=-c_{d-1}A^{d+i-1}-\cdots-c_1A^{1+i}-c_0A^i$ for $i \geq 0$. Induction confirms that $mA^n\in\M_d(\O_{\K})$ for all $n\geq 0$.

\medskip\noindent$(e) \Rightarrow (a)$ Suppose that $mA^n\in\M_d(\O_{\K})$ for all $n\geq 0$. Then
$\Lambda=\sum_{n\geq 0}A^n\O_{\K}^d$ is an $A$-invariant submodule of $m^{-1}\O_{\K}^d$ containing $\O_{\K}^d$.
Since $\O_{\K}$ is Noetherian \cite[Thm.~29.1, Sec.~15.3]{dummit2004abstract}, $\Lambda$ has finitely many generators $\vec{v}_1,\vec{v}_2,\ldots,\vec{v}_r$ over $\O_{\K}$. Write
$A\vec{v}_j=\sum_{i=1}^r c_{ij}\vec{v}_i$ with $C=[c_{ij}]\in\M_r(\O_{\K})$.
Let $V=[\vec{v}_1~\vec{v}_2~\ldots~\vec{v}_r]\in \M_{d\times r}(\K)$.
Then $AV=VC$, so $f(A)V=Vf(C)$ for every $f\in\K[x]$.
The Cayley--Hamilton theorem ensures $p_C(A)V=Vp_C(C)=0$.
Since $\Lambda$ contains $\O_{\K}^d$, the columns of $V$
span $\K^d$, so $p_C(A)=0$. Thus, $A$ is algebraic integral.

\medskip\noindent$(c) \Rightarrow (f)$ Apply Theorem \ref{Theorem:RatForm} with $B=\diag(E,Z)$.

\medskip\noindent$(f) \Rightarrow (c)$ Similar matrices share the same characteristic polynomial \cite[Thm.~10.3.1]{SCLA}. Since $A$ and $B$ are similar, $p_A(x)=p_B(x)\in\O_{\K}[x]$.

\medskip\noindent(a) $\Leftrightarrow$ (g) 
Since the eigenvalues of $A$ and $T$ are the same, use the equivalence (a) $\Leftrightarrow$ (d).
\end{proof}

\begin{cor}
Let $A,B\in \M_d(\K)$ be algebraic integral over $\K$ and $AB=BA$. Then $AB$ and $A+B$ are algebraic integral over $\K$.
\end{cor}

\begin{proof}
The eigenvalues of $A$ and $B$ are algebraic integers by Theorem \ref{Theorem:AlgebraicIntegral}. Since $AB=BA$, there are orderings $\lambda_1,\lambda_2,\ldots,\lambda_d$ and $\mu_1,\mu_2,\ldots,\mu_d$ of the eigenvalues of $A$ and $B$, respectively, such that $\lambda_1+\mu_1,\lambda_2+\mu_2,\ldots,\lambda_d+\mu_d$ are the eigenvalues of $A+B$ and $\lambda_1\mu_1,\lambda_2\mu_2,\ldots,\lambda_d\mu_d$ are the eigenvalues of $AB$ \cite[Cor.~11.5.2]{SCLA}. Sums and products of algebraic integers are algebraic integers, so the eigenvalues of $A+B$ and $AB$ are algebraic integers. Thus, Theorem \ref{Theorem:AlgebraicIntegral} ensures that $A+B$ and $AB$ are algebraic integral. 
\end{proof}

\begin{cor} \label{Kronecker Product}
Let $A\in \M_m(\K)$ and $B\in \M_n(\K)$ be algebraic integral over $\K$. Then $A\otimes B$ and $A\otimes I_n+I_m\otimes B$ are algebraic integral over $\K$. 
\end{cor}

\begin{proof}
The eigenvalues of $A$ and $B$ are algebraic integers by Theorem \ref{Theorem:AlgebraicIntegral}. Let $\lambda_1, \lambda_2,\ldots, \lambda_m$ and $\mu_1, \mu_2, \ldots, \mu_n$ be the eigenvalues of $A$ and $B$ respectively. The eigenvalues of $A\otimes B$ are $\lambda_i\mu_j$ for $i\in\{1,2,\ldots,m\}$ and $j\in\{1,2,\ldots,n\}$ and the eigenvalues of $A\otimes I_n+I_m\otimes B$ are  $\lambda_i+\mu_j$ for $i\in\{1,2,\ldots,m\}$ and $j\in\{1,2,\ldots,n\}$ \cite[Ex.~9.31]{SCLA}. Since the sums and products of algebraic integers are algebraic integers, the eigenvalues of $A\otimes B$ and $A\otimes I_n+I_m\otimes B$ are algebraic integers.  Theorem \ref{Theorem:AlgebraicIntegral} ensures that $A\otimes B$ and $A\otimes I_n+I_m\otimes B$ are algebraic integral.
\end{proof}

\section{Power-integral matrices}\label{Section:PowerIntegral}
We say that $A \in \M_d(\K)$ is \emph{power integral over $\K$} if $A^n \in \M_d(\O_{\K})$ for some $n\geq 1$.  Direct sums, transposes, and Kronecker products of power-integral matrices (over the same number field) are power integral.  The product of commuting power-integral matrices is also power integral.
If $A = T+N$ is the decomposition of Theorem \ref{Theorem:IndexNilpotent}, then $A^n = T^n$ for all $n\geq \ind A$.  Thus, $A$ is power-integral if and only if $T$ is.

\begin{ex}
Let $B\in \M_d(\Z)$ and $B^2=0$; suppose that the entries of $B$ are relatively prime.  If $n \geq 2$ then $A=I+\frac{1}{n}B$ is power integral over $\Z$ since  $A^n\in \M_d(\Z)$, but $A^k\notin \M_d(\Z)$ for $k=1,2,\ldots,n-1$ \cite{christofides2025powers}. 
The same phenomenon was independently observed in the context
of numerical semigroups \cite[Prop.~2.2.a]{NSRM1}.
\end{ex}

\begin{ex}
Let $S$ be a numerical semigroup; that is, an additive subsemigroup of $\Z_{\geq 0}$ with finite complement,
and let $d = \min (S \setminus\{0\})$.  Then there is an $A \in \M_d(\Q)$ such that 
$S = \{ n \in \Z_{\geq0} : A^n \in \M_d(\Z)\}$ \cite{NSRM2}; the dimension $d$ need not be optimal.
For example, $S = \{0,4,6,7,8,9,\ldots\}$ for $A = \minimatrixsmall{72785}{ \frac{14556}{3125} }{0}{5}$ \cite[Ex.~2.9]{NSRM4}.
\end{ex}

The next result is a broad generalization of \cite{PRM}, in which it is proved that if $A = SBS^{-1}$ with $B,S\in \M_d(\Z)$, $\det B = \pm 1$, and $m = \det S \neq 0$, then there is a positive $r \leq |m|^{2d^2}$ such that $A^r \in \M_d(\Z)$. Our result improves upon this in two ways.  First, it applies
to arbitrary number fields, as opposed to just the rational numbers.  
Second, we allow $A$ to be singular: instead of a unit determinant, we require its pseudo-determinant to be a unit in $\O_{\K}$ and the Drazin projection $A\hat A$ to have integral entries. For invertible matrices, our exponent bound is also smaller.
In what follows, $\O_{\K}^{\times}$ denotes the group of units in $\O_{\K}$.
We write $\N_{\K}(\alpha)=|\N_{\K/\Q}(\alpha)|$ for the absolute field norm of $\alpha\in\K$.

\begin{thm}\label{Theorem:Pigeon}
If $A = SBS^{-1}$, in which $B,S\in \M_d(\O_{\K})$, $\pdet A\in \O_{\K}^{\times}$,  $\alpha = \det S \neq 0$,
and $A \hat{A} \in \M_d(\O_{\K})$,
then there is a positive integer 
\begin{equation*}
    r\leq \max\{\ind A, 1\}\N_{\K}(\alpha)^{d^2}
\end{equation*}
such that $A^r\in \M_d(\O_{\K})$.
\end{thm}

Before proving this result, we remark that Theorem \ref{Theorem:Unit} below ensures that a power-integral matrix $A \in \M_d(\K)$ with $\pdet A \in \O_{\K}^{\times}$ is necessarily of the form $A = SBS^{-1}$, in which $B,S \in \M_d(\O_{\K})$ with $S$ invertible. The same theorem also ensures that $A\hat{A}$ is integral, so the additional hypothesis above is automatic for power-integral matrices.

\begin{proof}[Proof of Theorem \ref{Theorem:Pigeon}]
Let $k=\max\{\ind A,1\}$ and $P = B\hat{B}$. Since $\pdet B=\pdet A\in\O_{\K}^{\times}$,
Lemma \ref{Lemma:IntegralDrazin} ensures that $\hat{B} \in \M_d(\O_{\K})$.
Thus, $P\in\M_d(\O_{\K})$ is idempotent, $B^kP=B^k$, and $B^k\hat{B}^k=P$.
Let $\RR=\O_{\K}/\alpha\O_{\K}$. Since $|\RR|=\N_{\K}(\alpha)$, we have
\begin{equation*}
|\M_d(\RR)|=\N_{\K}(\alpha)^{d^2}.
\end{equation*}
The sequence of cosets $[B^{k}], [B^{2k}],\ldots$ takes values in this finite set,
so the pigeonhole principle provides $i > j \geq 1$ such that
\begin{equation*}
B^{ik}\equiv B^{jk}\pmod{\alpha\O_{\K}}
\end{equation*}
and $1\leq i-j\leq\N_{\K}(\alpha)^{d^2}$. Multiplying by $\hat{B}^{jk}$ gives
\begin{equation*}
B^{(i-j)k}P\equiv P \pmod{\alpha\O_{\K}}.
\end{equation*}
Since $B^{(i-j)k}P=B^{(i-j)k}$, there exists an $X\in\M_d(\O_{\K})$ such that
\begin{equation*}
B^{(i-j)k}=P+\alpha X.
\end{equation*}
Then
\begin{equation*}
A^{(i-j)k} = S(P+\alpha X)S^{-1} = A\hat{A}+ SX(\adj S) \in\M_d(\O_{\K}),
\end{equation*}
in which the last inclusion follows because $A\hat{A} \in \M_d(\O_{\K})$.
Taking $r=(i-j)k$ yields
\begin{equation*}
1\leq r\leq \max\{\ind A,1\}\N_{\K}(\alpha)^{d^2}. \qedhere
\end{equation*}
\end{proof}

\begin{ex}
The condition $A \hat{A} \in \M_d(\O_{\K})$ cannot be dropped.  Let $B = \minimatrixsmall{1}{0}{0}{0}$
and $S = \minimatrixsmall{1}{1}{0}{2}$.  Then $A = SBS^{-1} = \minimatrixsmall{1}{-1/2}{0}{0}$ satisfies $\pdet A = 1$, but $A^n = A$ is not integral for any $n \geq 1$.
\end{ex}

\begin{rem}
If $\N_{\K}(\alpha)>1$, the bound is strict. If $P \equiv 0 \pmod{\alpha\O_{\K}}$, then 
$B^k \equiv 0 \pmod{\alpha\O_{\K}}$.  Take $r=k$ so that $B^k = \alpha X$ and hence
$A^k = S X (\adj S) \in \M_d(\O_{\K})$.
Otherwise, none of the $B^{ik}$ with $i \geq 1$ is zero modulo $\alpha\O_{\K}$ since $B^{ik}\widehat{B^{ik}}=P$. The pigeonhole argument in the proof of Theorem \ref{Theorem:Pigeon} can be applied to a set of at most $\N_{\K}(\alpha)^{d^2}-1$ elements, which yields a smaller bound.
If $\N_{\K}(\alpha) = 1$, then $\alpha$ is a unit, so $S^{-1}\in \M_d(\O_{\K})$. 
Consequently, $A$ is integral and its least positive integral exponent is $1$.
Thus, the bound in Theorem \ref{Theorem:Pigeon} is attained precisely when
$\N_{\K}(\alpha) = 1$ and $\ind A \leq 1$.
\end{rem}

\begin{ex}\label{Example:ExplainFirst}
Consider the matrix $A \in \M_3(\K)$ in \eqref{eq:FirstExample}, in which $\K = \Q(\theta)$ with $\theta = \sqrt[3]{2}$.  
We can write $A=SBS^{-1}$, in which 
\begin{equation*}
    S = \begin{bmatrix}
        1 & 0 & 0\\
        0 & 91 & 1\\
        0 & 0 & 1
    \end{bmatrix}\quad \text{and}\quad 
    B = \begin{bmatrix}
        0 & 1 & 0\\
        0 & 92+\theta-\theta^2 & 1\\
        1 & -91(93+2\theta) & -92-\theta-\theta^2
    \end{bmatrix}.
\end{equation*}
Then $\alpha=\det S=91$. Since $\N_{\K}(\alpha)=91^3$, $\ind A = 0$, $[\K:\Q] = 3$, and $d = 3$, Theorem \ref{Theorem:Pigeon} ensures that $A^n\in \M_3(\O_\K)$ for some $n \leq (91^3)^{3^2}=91^{27}$.
Let $\RR = \O_{\K} = \Z[\theta]$ and observe that $\det A= 1$ and $p_A(x) = x^3 + 2 \theta^2 x^2 - (1+\theta^2)x - 1$.  For $n\geq 0$, the Cayley--Hamilton theorem ensures that $A^n = c_2(n)A^2 + c_1(n) A + c_0(n) I$, in which each $c_i(n) \in \RR$.  When entries are written in lowest terms, the only denominators that occur in $A$ and $A^2$ divide $91 = 7 \cdot 13$.  A brief computation verifies that $c_2(n) A^2 + c_1(n) A + c_0(n) I \in \M_3(\RR)$ if and only if $c_1(n), c_2(n) \in 91 \RR$. Let us work in the $\RR/91\RR$ algebra
$(\RR/91\RR)[x]/ \langle p_A(x) \rangle$.
Then $A^n \in \M_3(\RR)$ if and only if the image of $x^n = c_0(n) + c_1(n)x + c_2(n)x^2$ belongs to $\RR/91\RR$.
Since $\RR/91\RR \cong \RR / 7 \RR \times \RR / 13 \RR$, we work modulo $7$ and $13$ instead. The least exponents that work modulo $7$ and $13$ are $n_7 = 39{,}331$ and $n_{13} = 229{,}848$, respectively. This yields $n = \operatorname{lcm}(n_7,n_{13}) = 9{,}040{,}151{,}688$, as claimed in the introduction.  This is in accord with the bound afforded by Theorem \ref{Theorem:Pigeon}.
\end{ex}

\begin{thm}\label{Theorem:Unit}
Let $A\in \M_d(\K)$ with $\pdet A\in \O_{\K}^{\times}$. The following are equivalent.
\begin{enumerate}
\item $A$ is power integral over $\K$.
\item $m_A(x)\in \O_{\K}[x]$ and $A \hat{A} \in \M_d(\O_{\K})$.
\item $p_A(x)\in \O_{\K}[x]$ and $A \hat{A} \in \M_d(\O_{\K})$.
\item There exist a $B\in \M_d(\O_{\K})$ and an $S\in\M_d(\O_{\K})$ with $\det S \neq 0$ such that $A=SBS^{-1}$
and $A \hat{A} \in \M_d(\O_{\K})$.
\item $\hat{A}$ is power integral over $\K$.
\item $T = A^2\hat{A}$ is power integral over $\K$.
\end{enumerate}\medskip
In particular, these conditions imply that $A\hat{A} \in \M_d(\O_{\K})$.
\end{thm}

\begin{proof}
Let $A\in \M_n(\K)$ with $\pdet A\in \O_{\K}^{\times}$.

\medskip\noindent$(a) \Rightarrow (b) \Rightarrow (c) \Rightarrow (d)$ 
Since $A^n \in \M_d(\O_{\K})$ for some $n \geq 1$, the monic polynomial 
$p_{A^n}(x^n)\in\O_{\K}[x]$ annihilates $A$. Thus, $A$ is algebraic integral, and Theorem \ref{Theorem:AlgebraicIntegral} applies.
To see that $(a)$ implies $A\hat A\in\M_d(\O_{\K})$, suppose that $A^n\in\M_d(\O_{\K})$ for some $n\geq 1$.
Apply Lemma \ref{Lemma:IntegralDrazin} to $A^n$ and deduce that
$\widehat{A^n} = (\hat{A})^n \in \M_d(\O_{\K})$ since $\pdet(A^n)=(\pdet A)^n$ is a unit. Thus,
$A \hat{A} = A^n(\hat{A})^n \in \M_d(\O_{\K})$.

\medskip\noindent$(d)\Rightarrow (a)$ This follows from Theorem \ref{Theorem:Pigeon}.

\medskip\noindent$(a) \Leftrightarrow (e)$
In the notation of Theorem \ref{Theorem:RatForm},
$p_A(x)=x^r p_E(x)$. Thus,
\begin{equation*}
p_{\hat{A}}(x) =x^r p_{E^{-1}}(x) =(-1)^{d-r}(\det E)^{-1}x^d p_E(1/x).
\end{equation*}
Since $\pdet A = \det E\in\O_{\K}^{\times}$,
we have $p_A(x)\in\O_{\K}[x]$ if and only if $p_{\hat A}(x)\in\O_{\K}[x]$.
Moreover, 
\begin{equation*}
\hat{\hat{A}}=T=A^2\hat{A},
\quad \text{so} \quad
\hat{A}\hat{\hat{A}} =\hat{A} A^2\hat{A} =(A\hat{A})^2=A\hat{A}.
\end{equation*}
Thus, $A$ satisfies (c) if and only if $\hat{A}$ does. Since
$\pdet\hat{A}=(\pdet A)^{-1}\in\O_{\K}^{\times}$, the desired result follows from two applications
of $(a)\Leftrightarrow(c)$.

\medskip\noindent$(a) \Leftrightarrow (f)$ 
Simply note that $A^n = T^n$ for $n \geq \ind A$.

\medskip 
Suppose that any of the given statements hold.  Then $A$ and $\hat{A}$ are power integral, so there exist $i,j \geq 1$ such that $A^i, \hat{A}^j \in \M_d(\O_{\K})$.  Then $A \hat{A} = (A \hat{A})^{ij} = (A^i)^j (\hat{A}^j)^i \in \M_d(\O_{\K})$ since $A\hat{A}$ is idempotent and $A\hat{A} = \hat{A}A$.
\end{proof}

\begin{cor}
Let $A\in \M_d(\K)$ be power integral over $\K$ with $\det A \in \O_{\K}^{\times}$.
If $B\in \M_d(\K)$ is similar to $A$ over $\K$, then $B$ is power integral over $\K$.
\end{cor}

\begin{proof}
Since $B$ is similar to $A$, we have $p_B = p_A \in \O_{\K}[x]$ and
$\det B = \det A \in \O_{\K}^{\times}$. Theorem \ref{Theorem:Unit} applies to $B$ 
since $B\hat{B} = I$.
\end{proof}

\begin{ex}
The unit-determinant hypothesis cannot be omitted from the previous corollary.
Consider $\K = \Q$, $A = \minimatrixsmall{2}{0}{0}{1}$, and $B=\minimatrixsmall{2}{1/2}{0}{1}$.  Then $A$ is power integral but
$B^n = \minimatrixsmall{2^n}{(2^n-1)/2}{0}{1}$ is never an integer matrix.
\end{ex}

The implication (a) $\Rightarrow$ (e) can be made more precise; see Theorem \ref{Theorem:SemigroupInverse}.

\section{Numerical semigroups}\label{Section:Numerical}
In what follows, the term ``semigroup'' refers to an additive subsemigroup of $\Z_{\geq 0}$.
Each such semigroup is finitely generated by a unique generating set that is minimal with respect to containment \cite[Thm.~2.7]{Rosales}.
A \emph{numerical semigroup} is a semigroup with finite complement in $\Z_{\geq 0}$ \cite{Assi, Rosales}.  
Suppose that $\K$ is a number field with ring of integers $\O_{\K}$.  
For each $A \in \M_d(\K)$, the corresponding 
\emph{exponent semigroup} is the semigroup
\begin{equation}
\S_{\K}(A)=\{n\in\Z_{\geq 0} :  A^n\in\M_d(\O_{\K})\}.
\end{equation}
Observe that $\S_{\K}(A) \neq \{0\}$ if and only if $A$ is power-integral over $\K$.
Exponent semigroups were introduced in the setting $\K = \Q$ and $\O_{\K} = \Z$ in \cite{NSRM1} and further studied in \cite{NSRM2,NSRM3,NSRM4}.  All previous work on the topic focused exclusively on the rational field.

In the decomposition of Theorem \ref{Theorem:RatForm}, we have $Z^n = 0$ for $n \geq \ind A$.
Therefore, $\S_{\K}(A)$ and $\S_{\K}(T)$ can differ only for $1 \leq n < \ind A$.  We can say a bit more.

\begin{thm}\label{Theorem:Contain}
Let $A \in \M_d(\K)$ and $\pdet A \in \O_{\K}^{\times}$.  Then $\S_{\K}(A) \subseteq \S_{\K}(T)$.
\end{thm}

\begin{proof}
If $\S_{\K}(A) = \{0\}$, there is nothing to prove.
If $\S_{\K}(A) \neq \{0\}$, then $A$ is power integral over $\K$.
Theorem \ref{Theorem:Unit} ensures that the idempotent $A\hat{A}$ belongs to $\M_d(\O_{\K})$.
If $A^n \in \M_d(\O_{\K})$, then $T^n = (A^2 \hat{A})^n = A^n (A \hat{A})^n \in \M_d(\O_{\K})$.
Thus, $\S_{\K}(A) \subseteq \S_{\K}(T)$.
\end{proof}

Equality can occur in the previous theorem. For example, $\ind A = 1$ ensures that $A=T$ and hence $\S_{\K}(A) = \S_{\K}(T)$.  More generally, equality holds whenever $\S_{\K}(T)$ contains no positive integer less than $\ind A$, because $A^n=T^n$ for $n\geq\ind A$.

\begin{ex}\label{Example:ContainFail}
The containment $\S_{\K}(A) \subseteq \S_{\K}(T)$ may fail if $\pdet A \notin \O_{\K}^{\times}$.
Observe that $\pdet A = 2$, $\S_{\Q}(A) = \Z_{\geq 0}$, and $\S_{\Q}(T) = \{0,2,3,\ldots\}$, in which
\begin{equation*}
A = \megamatrix{2}{1}{0}{0}{0}{1}{0}{0}{0}
\quad \text{and} \quad 
T = \megamatrix{2}{1}{1/2}{0}{0}{0}{0}{0}{0}.
\end{equation*}
\end{ex}

\begin{thm}\label{Theorem:SemigroupInverse}
Let $A \in \M_d(\K)$ with $\pdet A \in \O_{\K}^{\times}$. 
Then $\S_{\K}(A) \subseteq \S_{\K}(\hat{A})$.
\end{thm}

\begin{proof}
The case $n=0$ is immediate. Suppose that $n \geq 1$ belongs to
$\S_{\K}(A)$, so that $B=A^n\in\M_d(\O_{\K})$. Since $\pdet B=(\pdet A)^n\in\O_{\K}^{\times}$,
Lemma \ref{Lemma:IntegralDrazin} ensures that
$\hat{B} = \widehat{A^n} = (\hat{A})^n \in \M_d(\O_{\K})$.
Thus, $n\in\S_{\K}(\hat A)$.
\end{proof}

\begin{cor}
Let $A \in \M_d(\K)$ with $\pdet A \in \O_{\K}^{\times}$. 
Then $\S_{\K}(\hat{A}) \subseteq \S_{\K}(T)$.    
\end{cor}

\begin{proof}
Since $T = \hat{\hat{A}}$ \cite[Thm.~4.11, p.~169]{BenIsrael}, this follows from Theorem \ref{Theorem:SemigroupInverse} applied to $\hat{A}$, whose pseudo-determinant is $(\pdet A)^{-1} \in \O_{\K}^{\times}$.
\end{proof}

\begin{ex}
Strict containment can occur in Theorems \ref{Theorem:Contain} and \ref{Theorem:SemigroupInverse}.  
Observe that $\pdet A = 1$, $\S_{\Q}(A) = \{0,2,3,\ldots\}$, and $\S_{\Q}(T) = \Z_{\geq 0}$, in which
\begin{equation*}
A = \megamatrix{1}{0}{0}{0}{0}{1/2}{0}{0}{0}
\quad \text{and} \quad 
T = \hat{A} = \megamatrix{1}{0}{0}{0}{0}{0}{0}{0}{0}.
\end{equation*}
\end{ex}

The \emph{matricial dimension over $\K$} of a semigroup $S \subseteq \Z_{\geq 0}$ is the smallest $d \geq 1$, denoted $\mdim_{\K} S$,
such that $S = \S_{\K}(A)$ for some $A \in \M_d(\K)$. Since $\mdim_{\Q} S \leq \min (S \setminus \{0\})$ for $S \neq \{0\}$ \cite[Cor.~3.6]{NSRM2}, the matricial
dimension is well defined over any number field.  Since $\K$ is a $[\K:\Q]$-dimensional vector space over $\Q$, 
the following holds:
\begin{equation*}
\mdim_{\K} S \leq \mdim_{\Q} S \leq [\K:\Q]\mdim_{\K}S.
\end{equation*}
Equality holds throughout if $\K = \Q$.
For the lower bound, equality occurs in other cases. For example, $\mdim_{\K} S = 1$
occurs if and only if $S=\{0\}$ or $\Z_{\geq 0}$.  Similarly, $\mdim_{\K} S = 2$
for the semigroups of \cite[Prop.~2.2]{NSRM1}.  Numerical work suggests the
next problem.

\begin{prob}
Does $\mdim_{\K} S = \mdim_{\Q} S$ for every semigroup $S$ and number field $\K$?
\end{prob}

The search for a counterexample is difficult since there is no known algorithm to compute $\mdim_{\Q} S$, let alone $\mdim_{\K}S$.
However, methods exist that work for certain special cases.  For example, $\mdim_{\Q} S$ has been computed 
for all $S$ with Frobenius number at most $10$ \cite{NSRM4}; recall that the \emph{Frobenius number} $F(S)$ of a numerical semigroup $S$ is the largest integer not in $S$. A more modest goal is the next problem.

\begin{prob}
Are the following equivalent:
(a) $\mdim_{\Q} S = 2$;
(b) $\mdim_{\K} S = 2$ for some number field $\K$;
(c) $\mdim_{\K}S = 2$ for every number field $\K$?
\end{prob}

Since (a) implies (c) and (c) implies (b), the only unresolved implication is (b) $\Rightarrow$ (a).

Semigroups of matricial dimension $2$ are characterized in terms of the $p$-adic valuations of certain Lucas sequences \cite{NSRM3}.  Extending this to the number-field setting would require analogues of such divisibility results, and these do not appear to exist in the literature \cite{BallotBook}.  Moreover, one wonders whether $\O_{\K}$ being a unique factorization domain comes into play.

\begin{prop}\label{Proposition:Strongly}
Let $A, B\in \M_d(\K)$ commute, and suppose that $\S_{\K}(A)$ and $\S_{\K}(B)$ are numerical semigroups.  Then $\S_{\K}(A+B) \neq \{0\}$.
\end{prop}

\begin{proof}
Let $f_1 = \max\{F(\S_{\K}(A)),0\}$ and $f_2 = \max\{F(\S_{\K}(B)),0\}$ (since the Frobenius number of $\Z_{\geq 0}$ is $-1$) and denote by
$e_1$ and $e_2$ the smallest positive integers such that $e_1 A^i \in \M_d(\O_{\K})$ for $1 \leq i \leq f_1$
and $e_2 B^j \in \M_d(\O_{\K})$ for $1 \leq j \leq f_2$.  
Since $e_i\geq2$ whenever $f_i\geq 1$, we have $e_if_i!\geq f_i+1$, and hence $n>f_1+f_2$.
Now check that $(A+B)^n \in \M_d(\O_{\K})$, in which $n = e_1e_2f_1!f_2!$, using the binomial theorem.
\end{proof}

\begin{cor}
If $A\in \M_n(\K)$ is nilpotent, then $mI_n+A$ is power integral for all $m\in \O_{\K}$.
\end{cor}

\begin{cor}
Let $A \in \M_m(\K)$ and $B \in \M_n(\K)$.
If $\S_{\K}(A)$ and $\S_{\K}(B)$ are numerical semigroups, then $\S_{\K}(A\otimes I_n+I_m\otimes B) \neq \{0\}$.
\end{cor}

\begin{proof}
Observe that $\S_{\K}(A\otimes I_n)=\S_{\K}(A)$ and $\S_{\K}(I_m\otimes B)=\S_{\K}(B)$. Since
$(A\otimes I_n)(I_m\otimes B)=(AI_m)\otimes (I_nB)= (I_mA)\otimes(BI_n)=(I_m\otimes B)(A\otimes I_n)$,
Proposition \ref{Proposition:Strongly} yields the assertion.
\end{proof}

\bibliography{Main}
\bibliographystyle{amsplain}

\end{document}